\documentclass{article}
\usepackage[utf8]{inputenc}
\usepackage{amsmath,amsfonts,amssymb}
\usepackage{amsthm}
\usepackage[noadjust]{cite}
\usepackage{enumitem}
\newtheorem{definition}{Definition}

\newtheorem{theorem}{Theorem}

\title{Jackson network in a random environment: strong approximation}
\author{Elena Bashtova\footnote{Lomonosov Moscow State University; supported by RFBR grant 20-01-00487; e-mail: elena.bashtova@math.msu.ru}, Elena Lenena\footnote{Lomonosov Moscow State University} }
\date{}

\begin{document}


\maketitle{}



\begin{abstract}

We consider a Jackson network with regenerative input flows in which every server is subject to a random environment influence generating breakdowns and repairs. They occur in accordance with two  independent sequences of i.i.d. random variables. We establish a theorem on the strong approximation of the vector of queue lengths by a reflected Brownian motion in positive orthant.
\end{abstract}




\pagestyle{empty}

\section{Introduction}
\label{S:1}
Jackson networks are one of the most fundamental objects in the theory of stochastic processing networks. 
Such network models have long been used for a wide range of applications in transportation, production, computer science, social networks etc., so results concerning their asymptotic behavior present a practical interest.

There are multiple research directions: one is to evaluate the limit distribution and its product forms in the case when this limit distribution exists. Another research direction considers systems with heavy traffic.       
Heavy traffic and overloaded system cases are analytically non-trivial and at the same time the most crucial, as overloading may lead to breakdowns, production shutdown, losses from downtime, repair costs and other overheads. 

Heavy traffic limit theorems of the type considered in this paper have attracted a lot of attention previously.
 Harrison \cite{harrison} considered tandem queueing systems and proved a heavy traffic limit theorem for the stationary distribution of the sojourn times. His limit was also given as a complicated function of multidimensional Brownian motion. Harrison later again considered tandem queueing systems, and introduced reflected Brownian motion on the non-negative orthant as the diffusion limit. Other work in the areas of heavy traffic limit theorems and diffusion approximations is surveyed in Whitt  \cite{whitt} and Lemoine \cite{lemoine}. Reiman \cite{reiman} presented heavy traffic limit theorems for the queue length and sojourn time processes associated with open queueing networks. These limit theorems state that properly normalized sequences of queue length and sojourn time converge to a reflected Brownian motion in an orthant. Furthermore, strong (almost sure) approximation with infinite time horizons and a reliable server networks were considered by Chen and Yao \cite{chen}.
There has been also a growing literature on queues with an unreliable server. To emphasize the significance of queueing models with unreliable servers in applications we will refer to some works \cite{djellab}, \cite{gaver}, \cite{sherman}, \cite{kalimulina}.

The system under consideration in \cite{coupling} is a single queue system with an unreliable server, and it emerged  as a mathematical model of unregulated crossroads. The simplest model of a vehicle crossing problem in probabilistic terms was considered in \cite{harreim}. There is a one-lane road $S_{1}$ which is intersected on one side by a single-lane secondary road $S_{2}$. A car waiting on the secondary road $S_{2}$ will turn right only if there is at least distance $J$ between the intersection and the first car on $S_{1}$. We may consider the crossroads with respect to cars arriving on the secondary road $S_{2}$ as a queueing system with an unreliable server. The server is in working state when there are no cars within a distance $J$ of the crossroads on the road $S_{1}$ and it becomes out of order when the first car appears on this interval. By the nature of the system we have to suppose that the breakdown of the server can occur at any time, even when the server is free. Note that a queueing system with an unreliable server may also be considered in the stochastic analysis of crossroads with traffic lights.

In this paper we present a strong Gaussian approximation for a more functionally complicated system with unreliable servers and generalized input flows. Additionally, we do not require input flows to be independent. 

\begin{definition}
We say that a vector-valued random process $\zeta =\{\zeta_t,\, t\geq 0\}$ admits a $r$-strong approximation by  some process $\overline{\zeta}=\{\overline{\zeta}_t,t\geq 0\}$,
if there exists a probability space $(\Omega,\mathcal{F},
\mathsf{P})$ on which one can define both $\zeta$ and $\overline{\zeta}$ in such a way that
$$\sup\limits_{0\le u\le t}\|\zeta_u-\overline{\zeta}_u\| = o(t^{1/r}),\text{ a.s., when }t\to\infty. $$

\end{definition}



\section{System description}
\label{S:2}

\subsection{The queuing network with unreliable servers}

The queuing network $\mathcal N$ we study has $K$ single server stations, and each of them has an associated infinite capacity waiting room. At least one station has an arrival stream from outside the network, and the vector $A(\cdot)$ of arrival streams is assumed to be a multi-dimensional regenerative flow. Recall

\begin{definition}

A multi-dimensional coordinatewise c\`{a}dl\`{a}g stochastic process $A (t)$ is called regenerative one if there exists an increasing sequence of random variables $\{\theta_{i}, i \geq 0\}$, $\theta_{0} = 0$  such that the sequence
$$\{\kappa_i\}_{i=1}^{\infty} = \{ A(\theta_{i-1}+t)- A(\theta_{i-1}),\theta_{i} - \theta_{i-1}, t \in [0, \theta_{i} - \theta_{i-1})\}_{i=1}^{\infty}$$
consists of independent identically distributed  random elements.
The random variable $\theta_{j}$ is said to be the $j$th regeneration point of  $A,$ and $\tau_{j}= \theta_{j} - \theta_{j-1}$ $($where $\theta_{0} = 0)$ to be the $j$th regeneration period. 
\end{definition}

From Smith \cite{Smith1955} it is known that we can  define the asymptotic intensity vector $\lambda = \lim\limits_{t\to\infty}\frac{A(t)}{t}$ a.s., and the asymptotic covariance matrix $V = \lim\limits_{t\to\infty}\frac{\mathsf{Var}A(t)}{t}$.

Each station's service times $\{\eta_{j}^{i}\}_{i=1}^{\infty}$, $j=1,\ldots, K,$ are mutually independent sequences of i.i.d. random variables with mean $1/\mu_j$ and variance $\sigma^2_j$. After being served at station $k,$ the customer is routed to station $j$ $(j=1,\ldots,K)$ with probability $p_{kj}$.
The routing matrix $P=\|p_{kj}\|_{k,j=1}^K$ 
is assumed to have spectral radius strictly less than unity, i.e. there is always a positive probability than a served customer leaves the system immediately.
For  $i=1,2,\ldots$ and $k=1,\ldots K$, define $\varphi_{k}^{i}$ to be a random variable equal to $j = 1,\dots, K$ whenever $i$th customer served on station $k$ is routed to station $j,$ and $\varphi_{k}^{i} = 0$ if this customer exits the network.
Routing vectors are defined by $\phi_{k}^{i}=e_{\varphi_{k}^{i}}$, where $e_{j}$ is the $K$-vector whose $j$th component is 1 and other components are 0, if $j=1,\ldots,K$, and  $e_0=0$. 
Customers routing happens independently and immediately.

In our model the service on every channel is influenced by  a random environment which causes breakdowns of the server  (falling into OFF state from ON state) at random moments. The repair of the server also takes a random time.
We suppose that consecutive  time intervals of states ON and OFF form two independent sequences of i.i.d. random variables and, for $j$th server, denote them  by $\{u_j^{i}\}_{i=1}^\infty$ and  $\{v_j^{i}\}_{i=1}^\infty$ respectively. Let $a_j = \mathsf{E}u_j^{1} $, $b_j = \mathsf{E}v_j^{1}$, $\alpha_j = {a_j}({b_j+a_j})^{-1}$,  $s^2_j = \mathsf{Var}u_j^{1} $, $d^2_j = \mathsf{Var}v_j^{i}$ for $j=1,\ldots,K$.
 We suppose that the service that was interrupted by the breakdown is continued upon repair from the point at which it was interrupted. 

Let $Q(t) = (Q_1(t),\dots, Q_K(t))$  be the vector of number of customers in each channel at time $t$. 
Next we need  vector-valued busy time processes $B(t) = (B_1(t), \dots, B_K(t))$.  The $j$th component of  $B(t)$, where $j=1,\ldots,K,$ indicates the amount of time up to $t$ the server at station $j$  is busy (i.e. in state ON and serving jobs). 
Then,

$$Q(t)= A(t)+\sum\limits_{j=1}^{K} {L_{j}(B_{j}(t))},$$
where 
$$ L_{j}(u) = \sum\limits_{i=1}^{S_{j}(u)} (\phi_{j}^{i}-e_{j}),\quad 1\leq k \leq K. $$

Next, we consider a system of equations 
\begin{equation}
\label{eq:tau}
\gamma_{j}=\lambda_{j}+ \sum_{i=1}^{k} (\gamma_i\wedge\alpha_i\mu_i) p_{ij},\quad j=1,\ldots,K.
\end{equation}

In the Jackson networks theory, the systems like above play a significant role and are known as the traffic equations. Due to Theorem 7.3 in Chen and Yao \cite{chen} our traffic equation (\ref{eq:tau}) has a unique solution $\gamma = (\gamma_{1},\dots, \gamma_{K})$. 
Therefore we may define a $j$th station traffic coefficient $\rho_j= \frac{\gamma_j}{\alpha_j\mu_j}$, $j=1,\ldots,K$. 
Buffer $j$ is called a nonbottleneck if $\rho_j<1$, a bottleneck if $\rho_j\ge1$, a balanced bottleneck if $\rho_j=1$, and a strict bottleneck if $\rho_j>1$.

\subsection{Reflected Brownian motion in orthant}
In order to articulate the key result we will introduce one more definition.
Consider a pair of $K$-dimensional processes $Z=\{ Z(t); t\geq 0\}$ and $Y=\{ Y(t); t\geq 0\}$  which jointly satisfy the following conditions: 

$$Z(t)= W(t)+Y(t)(I-P),t\geq 0,$$ 
where $W=\{ W(t); t\geq 0\}$ is a $K$-dimensional Brownian motion with covariance matrix $ \Gamma$, drift vector $b$ and $W(0)\in \mathbb{R}^K_+$; 
$Z(t)$ takes values in $\mathbb{R}^K_+$, $t\geq 0$;
$Y_{j}(.)$ is continuous and nondecreasing with 
$Y_{j}(0)=0$; 
and $Y_{j}(.)$ increases only at those times $t$ where 
$Z_{j}(t)=0,$ $j=1,\dots K$.
It was shown in \cite{harreim} that for any given Brownian motion $W$ there exists a unique pair of processes $Y$ and $Z$ satisfying conditions above. 

In the language of \cite{harreim}  and \cite{chen}, $Z$ is a reflected Brownian motion on $\mathbb{R}^K_+$ with drift $b$, covariance matrix $\Gamma$, and reflection matrix $(I-P)$.

\section{Main result}
\label{S:3}

\begin{theorem}
Let $\mathsf{E}\tau_i^p<\infty$, $\mathsf{E}\|A(\theta_{i+1}) - A(\theta_i)\|^p<\infty$, $\mathsf{E}(\eta^i_j)^p<\infty$ for $i\in\mathbb{N}$ and any $j=1,\ldots,K$. Then $Q$ admits $p'$-strong approximation by  a reflected Brownian motion on $\mathbb{R}^K_+$ with drift $\lambda-\alpha\mu(I-P)$, reflected matrix $(I-P)$, covariance matrix $\Gamma = ||\Gamma_{kl}||_{k,l = 1}^K$,
$$
\Gamma_{kl} =V_{kl} +\sum\limits_{j=1}^K (\gamma_j\wedge\alpha_j\mu_j)p_{jk}(\delta_{kl} - p_{jl})+ (\sigma^2_j\mu^3_j\alpha_j +\mu_j^2D_j)(\rho_j\wedge 1)(p_{jk}-\delta_{jk})(p_{jl} - \delta_{jl}) ,
$$
where
$$
D_j = \frac{a^2_{j} d^2_{j}+b^2_{j} s^2_{j}}{(a_{j}+b_{v^j})^3}
$$
and $p'=p$ for $p<4$ and $p'$ is any number less than $4$ for $p\ge 4$.
\end{theorem}
\begin{proof}
To begin with introduce a number of auxiliary processes. Let $\chi(t)$, $\overline{\chi}(t)$  indicate that at  time  $t$ server is in ON and OFF respectively,
$$S_j(t) = \max\{k: \sum\limits_{i=1}^k\eta_j^i\le t\},\quad
N_j(t) = \max\{k: \sum\limits_{i = 1}^k(u^{i}_j+v^{i}_j)\le t\},$$ 
$$C_{j}(t)=\sum\limits_{i=1}^{N_{j}(t)} u^i_{j}+\chi(t)\left(t-\sum_{i=1}^{N_{j}(t)} (u^{i}_j+v^{i}_j)\right) + \overline{\chi}(t) u^{N_j(t)+1}_j. $$ 
So, $C_j(t)$ equals the total time before $t$  $j$th server being in the state ON.

Now consider a process
$\widetilde{S}_j(t) = S_j(C_j(t))$ and a new network $\widetilde{\mathcal N}$ which is defined in the same way as initial network $\mathcal N$ with the following difference. Service times $\widetilde{\eta}_j^i$, $i=1,2,\dots$ for $j$th server are defined as
$$
\widetilde{\eta}^i_j = \inf\{t: \widetilde{S}_j(t) = i\} - \inf\{t: \widetilde{S}_j(t) = i-1\},\quad i=1,\dots.
$$
Note that random variables $\widetilde{\eta}^i_j$, $i=1,2,\dots,$ in general are not independent or identically distributed but, as one will see below, the network $\widetilde{\mathcal N}$ is more convenient to handle. 

For the network $\widetilde{\mathcal N}$, we denote the vector of number of customers at time $t$ and the vector of busy times up to time $t$  by $\widetilde{Q}(t)$ and $\widetilde{B}(t)$ respectively. 
Then
$$\widetilde{Q}(t)= A(t)+\sum\limits_{j=1}^{K} { \widetilde{L}_{j}(B_{j}(t))},\,\,\text{where}\,\, \widetilde{L}_{j}(u) = \sum\limits_{i=1}^{\widetilde{S}_{j}(u)} (\phi_{j}^{i}-e_{j}),\quad 1\leq j \leq K. $$


So, due to Theorem 1 in Bashtova, Shashkin \cite{bashtshash}  the vector-valued process $C=(C_1,\ldots,C_K)$ admits $p$-strong approximation by a Wiener process  $W^C$ whose components are independent  with drifts $c_j = \alpha_j$ and variances $D_j$, $j=1,\dots,K$. In particular $C$ satisfies the functional law of the iterated logarithm.   Also, since each $S_j$ is a renewal process ($j=1,\dots ,K$), by Cs\"{o}rg\H{o}, Horv\'{a}th and Steinebach \cite{CHS}  the process $S$ also admits  $p$-strong approximation by a Wiener process $W^S$ with independent components having drifts $\mu_j $ and variances $\sigma_j^2\mu_j^3$. Hence by Lemma 6.21 in Chen, Yao \cite{chen} the process $\tilde{S}$ admits $p$-strong approximation by a Wiener process with independent components having drifts $\mu_j\alpha_j$ 
and variances $\mu_j^2D_j+\sigma_j^2\mu_j^3\alpha_j,$ $j=1,\ldots,K.$ Additionally $A$ admits a $p$-strong approximation by  a $K$-dimensional Wiener process $W_0$ with drift $\lambda$ and covariance matrix $V$, and we can assume $W_0$ to be independent of all approximating Wiener processes above.

The Remark 7.17 of Chen, Yao \cite{chen} is applicable to both Theorems 7.13 and 7.19 there. Thus as by our construction the network $\tilde{N}$ satisfies the conditions (7.65), (7.66) and (7.68) of the latter theorem,  its first statement  implies the desired result.

\end{proof}











\end{document}